\RequirePackage{fix-cm}
\documentclass[smallextended]{svjour3}       
\smartqed  
\usepackage{graphicx}
\usepackage{tikz}

\renewcommand{\th}{^\mathrm{th}}
\journalname{Graphs and Combinatorics}
\begin{document}

\title{Packing tree degree sequences
}


\author{Aravind Gollakota        \and
           William Hardt \and
           Istv\'an Mikl\'os
}


\institute{A. Gollakota \at
              Cornell University \\
              144 East Avenue \\
              Ithaca, NY, USA 14853 \\
              \email{apg77@cornell.edu}           
           \and
           W. Hardt \at
              Carleton College \\ 300 North College Street \\ Northfield, MN, USA 55057 \\ \email{hardtw@carleton.edu}
           \and
            I. Mikl\'os \at
            R\'enyi Institute\\1053 Budapest, Re\'altanoda u. 13-15\\Hungary\\
             \email{miklos.istvan@renyi.mta.hu}\\
            \emph{Secondary affiliation:} SZTAKI\\1111 Budapest, L\'agym\'anyosi u. 11\\Hungary
}

\date{Received: date / Accepted: date}

\maketitle

\begin{abstract}
We consider packing tree degree sequences in this paper. We set up a conjecture that any arbitrary number of tree degree sequences without common leaves have edge disjoint tree realizations. This conjecture is known to be true for $2$ and $3$ tree degree sequences. In this paper, we give a proof for $4$ tree degree sequences and a computer aided proof for $5$ tree degree sequences. We also prove that for arbitrary $k$, $k$ tree degree sequences without common leaves and at least $2k-4$ vertices which are not leaves in any of the trees always have edge disjoint tree realizations. The main ingredient in all of the presented proofs is to find rainbow matchings in certain configurations.
\end{abstract}

\section{Introduction}
The colored degree matrix problem, \cite{hmcd2016} also known as finding edge packing \cite{bushetal2012}, edge disjoint realizations \cite{guinezetal2011} or degree constrained edge partitioning \cite{bentzetal2009}, asks if a series of degree sequences have edge disjoint realizations. The general problem is known to be NP-complete \cite{guinezetal2011}, but certain special cases are easy. These special cases include the case when one of the degree sequences is almost regular and there are only two degree sequences \cite{kundureg}, or, equivalently, when the sum of two degree sequences is almost regular \cite{chen1988}, or when the degrees are sparse \cite{bentzetal2009,hmcd2016}. Kundu proved that two degree sequences of trees have edge disjoint tree realizations if and only if their sum is graphical \cite{kundutree}. On the other hand, this is not true of three such sequences: there exist $3$ degree sequences of trees such that any couple of them have a sum which is graphical; furthermore, the sum of the $3$ is still graphical, and they do not have edge disjoint tree realizations \cite{kundu3tree}. However, $3$ degree sequences of trees do have edge disjoint tree realizations when the smallest sum of the degrees is $5$ \cite{kundu3tree}. This minimal degree condition includes the case when the degree sequences have no common leaves.

It is easy to see that the sum of two degree sequences of trees is always graphical if they do not have common leaves. Therefore $k$ degree sequences of trees always have edge disjoint tree realizations if they do not have common leaves for $k=2, 3$. A natural question is to ask if this statement is true for arbitrary $k$. In this paper we conjecture that it is true for arbitrary $k$, and prove it for $4$ degree sequences of trees.
For $5$ degree sequences of trees, we prove that the conjecture is true if it is true up to $18$ vertices. Computer aided search then confirmed that it is indeed true up to $18$ vertices.
We also prove the conjecture for arbitrary $k$ in a special case, 
when there are a prescribed number of vertices which are not leaves in any of the degree sequences. All the presented proofs are based on induction, and the key point in the inductive steps is to find rainbow matchings in certain configurations. 

\section{Preliminaries}
In this section, we give the necessary definitions and notation, as well as
state the conjecture that we prove for some special cases.

\begin{definition}
\begin{sloppypar}
A \emph{degree sequence} is a list of non-negative integers, ${D = d_1, d_2, \ldots d_n}$. A degree sequence $D$ is \emph{graphical} if there exists a graph $G$ whose degrees are exactly $D$. Such a graph is a \emph{realization} of $D$.
A degree sequence $D = d_1,d_1\ldots d_n$ is a \emph{tree degree sequence} if all degrees are positive and $\sum_{i=1}^{n} d_i = 2n-2$. A degree sequence is a \emph{path degree sequence} if two of its degrees are $1$ and all other degrees are $2$.
\end{sloppypar}
\end{definition}
It is easy to see that a tree degree sequence is always graphical and there is a tree realization of it.
\begin{definition}
A \emph{degree matrix} is a matrix of non-negative integers. A degree matrix $M$ of dimension $k \times n$ is \emph{graphical} if there exists a series of edge disjoint graphs, $G_1, G_2, \ldots, G_k$ such that for each $i$, $G_i$ is a realization of the degree sequence in the $i{\th}$ row. Such a series of graphs is a \emph{realization} of $M$. Alternatively, an edge colored simple graph is also called a realization of $M$ if it is colored with $k$ colors and for each color $c_i$, the subgraph containing the edges with color $c_i$ is $G_i$. The degree matrix might also be defined by its rows, which are degree sequences $D_1, D_2, \ldots D_k$. The degree of vertex $v$ in degree sequence $D_i$ is denoted by $d_v^{(i)}$.
\end{definition}

In this paper, we consider the following conjecture.
\begin{conjecture}\label{conj:main}
Let $D_1, D_2, \ldots, D_k$ be tree degree sequences without common leaves, that is, for any $v$ and $i$, $d_v^{(i)} = 1$ implies that for all $j\ne i$, $d_v^{(j)} >1$. Then they have edge disjoint realizations.
\end{conjecture}

A trivially necessary condition for a series of degree sequences be graphical is that the sum of the degree sequences is graphical. Therefore, Conjecture~\ref{conj:main} implies that sum of tree degree sequences without common leaves is always graphical. This implication is in fact true, and is proven below. Before proving it, we also prove a lemma which is interesting on its own.

\begin{lemma}\label{lem:2k-eg}
Let $F = f_1 \ge f_2 \ge \ldots \ge f_n$ be the sum of $k$ arbitrary tree degree sequences. Then the Erd{\H o}s-Gallai inequality 
\begin{equation}
\sum_{i = 1}^s f_i \le s(s-1) + \sum_{j=s+1}^n \min\{s, f_j\}
\end{equation}
holds for any $s \ge 2k$.
\end{lemma}
\begin{proof}
 For any sum of $k$ tree degree sequences,
\begin{equation}
\sum_{i = 1}^s f_i \le k(2n-2) - (n-s)k,
\end{equation}
since each tree degree sequence has a sum $2n-2$ and the minimum sum is $k$ on any vertex. Furthermore if $s \ge 2k$, then
\begin{equation}
s(s-1) + (n-s)k \le s(s-1) + \sum_{j=s+1}^n \min\{s,f_j\}.
\end{equation}
Therefore, it is sufficient to prove that
\begin{equation}
k(2n-2) - (n-s)k \le s(s-1) + (n-s)k. 
\end{equation}
Rearanging this, we get that
\begin{equation}
2k(s-1) \le s(s-1),
\end{equation}
which is true when $s\ge 1$ and $2k \le s$.
\qed
\end{proof}

We use this lemma to prove the following theorem -- now on tree degree sequences without common leaves.

\begin{theorem}
Let $D_1, D_2, \ldots, D_k$ be tree degree sequences without common leaves. Then their sum is graphical.
\end{theorem}
\begin{proof}
Let $F = f_1, f_2, \ldots, f_n$ denote the sum of the degrees in decreasing order. We use the Erd{\H o}s-Gallai theorem \cite{eg1960}, which says that a degree sequence $F$ in decreasing order is graphical if and only if for all $s$ with $1 \leq s \leq n$,
\begin{equation}
\sum_{i = 1}^s f_i \le s(s-1) + \sum_{j=s+1}^n \min\{s, f_j\}.\label{eq:eg}
\end{equation}
According to Lemma~\ref{lem:2k-eg}, it is sufficient to prove the inequality for $s \le 2k-1$, since for larger $s$, the inequality holds.

Since there are no common leaves, any  $f_j$ is at least $2k-1$, therefore  $\min\{s,f_j\}$ is $s$ for any $s \le 2k-1$. Writing this into Equation~\ref{eq:eg}, we get that
\begin{equation}
\sum_{i = 1}^s f_i \le s(s-1) + (n-s)s = s(n-1).
\end{equation}
And this is in fact the case since we claim that the sum of the degrees cannot be more than $n-1$ on any vertex. Indeed, $d_v^{(i)} = l$ means at least $l$ leaf vertices which are not $v$ in tree $T_i$. Since there are no common leaves, and there are $n-1$ vertices when $v$ is excluded, $f_v = \sum_{i=1}^k d_v^{(i)} \le n-1$. So this inequality holds for $s \leq 2k-1$.
\qed
\end{proof}

We now present partial results on Conjecture~\ref{conj:main}. The results are obtained by inductive proofs in which larger realizations are constructed from smaller realizations. The constructions use the existence of rainbow matchings, defined below.

\begin{definition}
A \emph{matching} is a set of disjoint edges. In an edge-colored graph, a \emph{rainbow matching} is a matching in which no two edges have the same color.
\end{definition}

\section{The theorem for $4$ tree degree sequences}

In this section, we are going to prove that $4$ tree degree sequences always have edge disjoint realizations if they do not have common leaves. The proof is based on induction. In the inductive step, we need the following lemma to reduce the case to a case with fewer number of vertices.

\begin{lemma}\label{lem:reduction}
Let $D^{(1)},D^{(2)},\ldots,D^{(k)}$ be tree degree sequences without common leaves such that not all of them are degree sequences of paths. Then there exists vertices $v$, $w$ and an index $i$ such that $d_v^{(i)} = 1$, $\forall\ j \ne i$ $d_v^{(j)} =2$,  and $d_w^{(i)} >2$.
\end{lemma}
\begin{proof}
We structure the proof in terms of the degree matrix and certain submatrices. To start with, let the degree matrix be structured such that the $v\th$ column corresponds to vertex $v$ and the $i\th$ row to degree sequence $i$, so that entry $(i,v)$ of the matrix is $d_v^{(i)}$. We will refer to vertices and columns interchangeably, and a set of vertices $S$ determines a set of columns that can be seen as a submatrix (but each submatrix does not always determine a set of vertices). We will use $|S|$ to denote the cardinality of the set, i.e.\ the number of columns of the submatrix $S$, and $\overline{M}$ to denote the average of the elements of a submatrix $M$.

Let $A$ be the set of ``low-degree'' vertices, $A = \{ v \mid d_v = 2k-1 \}$, and $B$ the set of ``high-degree'' vertices, $B = \{ v \mid d_v \geq 2k \}$. Note that this forms a partition of the vertex set $V$; there are no vertices $v$ with $d_v < 2k-1$ since that would force at least two common leaves. Note also that $A$ is nonempty; if it weren't then we would have $B=V$, which is impossible since the degrees of vertices in $B$ are too high: the total degree sum is at least $2kn$ when it should of course be exactly $2k(n-1)$. Similarly $B$ is nonempty as well, since not all the degree sequences are degree sequences of paths.

Observe that the vertices in $A$ must consist of exactly one leaf and the rest degree-2 vertices. That is, for all $v \in A$ there exists exactly one $i$ such that $d_v^{(i)} = 1$, and for all other $j$ we have $d_v^{(j)} = 2$; otherwise they would have common leaves.

The proof is by contradiction. Assume that for every pair $v \in A, w \in B$, in every sequence $i$ it is the case that $d_w^{(i)} \leq 2$ whenever $d_v^{(i)} = 1$.

Permute the rows and columns such that the columns of $A$ are to the left and those of $B$ to the right, where $A$ (considered as a submatrix) is ordered such that the rows with 1s are all on top. These rows determine submatrices $A'$ of $A$ and $B'$ of $B$, and the other rows determine submatrices $A''$ and $B''$ similarly, sitting at the bottom. We make some important observations:
\begin{enumerate}
\item Each row in $A'$ contains at least one 1 by construction (this was precisely why we picked those rows).
\item $A''$ consists entirely of 2s. This is again by construction, since we placed all rows with 1s on top in $A'$. In particular, $\overline{A''} = 2$.
\item Each element of $B'$ is at most 2. This follows from our assumption: for each row $i$ of $B'$ we know that $d_v^{(i)} = 1$ for all $v \in A$, and so by assumption $d_w^{(i)} \leq 2$ for all $w \in B$. In particular, $\overline{B'} \leq 2$.
\item $A''$ (and hence $B''$) is nonempty. If not then take any vertex $w$ having a degree greater than $2$. Such a vertex exists since not all degrees sequences are path degree sequences. If $i$ is an index for which $d_w^{(i)} > 2$, take a $v \in A'$ such that $d_v^{(i)} = 1$. Again, such vertex exists if $A''$ is empty.
\end{enumerate}

Now consider the submatrix $\left[\begin{array}{cc} A'' & B'' \end{array}\right]$, i.e.\ all the bottom rows. The row sum for each row must be $2n-2$, and since $\overline{A''} = 2$ this forces $\overline{B''} < 2$. But now consider the submatrix $B$, i.e.\ $\left[\begin{array}{c} B' \\ B'' \end{array}\right]$. The column sum in each column must be at least $2k$. Since  $\overline{B'} \leq 2$, this forces $\overline{B''} \geq 2$. This is a contradiction. \qed
\end{proof}

We also need the following lemma to be able to build up edge disjoint realization of a tree degree sequence quartet from the realization of a smaller tree degree sequence quartet.

\begin{lemma}\label{lem:rainbow}
Let $H = (V,E)$ be an edge colored graph in which $|V| \ge 10$ and the number of colors is $4$. Furthermore, the subgraphs for each color is a tree on n vertces, and the trees do not have a common leaf. That is, if $d_v^{(i)} = 1$, then $d_v^{(j)} >1$ for all $j\ne i$. Let $v_j \in V$ be an arbitrary vertex, and let $G$ be the subgraph of $H$ containing the first $3$ colors. Then $G \setminus \{v_j\}$ contains a rainbow matching of size $3$.
\end{lemma}
\begin{proof}
We will call the $3$ colors red, blue and green. Let $v$ be a vertex with at least one edge of each color going to a vertex that is not $v_j$; such a vertex can easily be seen to exist. Indeed, just pick any vertex adjacent to $v_j$ in the fourth tree. Thus let vertices $v, u_1, u_2, u_3$ be such that $(v, u_1)$ is blue, $(v, u_2)$ green and $(v, u_3)$ red. We will refer to these vertices as ``the complex''.

Let $W = \{w_1, \dots, w_5\}$ be five other vertices; these exist because $|V| \geq 10$. We make some important observations. Let $w$ be an arbitrary vertex in $W$. Because we have no common leaves, vertex $w$ has degree at least 2 in all colors with the possible exception of one, in which it may have degree 1; in any case it has total degree at least 5. It may be adjacent to $v_j$ in some color, so excluding $v_j$ it has total degree at least 4. Moreover, for each color, $w$ is incident with at least two edges not of this color and avoiding $v_j$.

We claim that we can always find a rainbow matching in this setup. Our proof will have three cases based on how many disjoint edges there exist within $W$: at least two, exactly one, and none (i.e.\ no edges within $W$ at all).

\paragraph{Case 1: At least two disjoint edges within $W$.}
Let two disjoint edges be $(w_1, w_2)$ and $(w_3, w_4)$. Clearly we may assume they are the same color because otherwise we have a rainbow matching easily; say they are red. Now observe that $w_5$ has at least 2 non-red edges, and they must both go to the complex because if there was one edge, $e$, which didn't, it would cover only one of the two disjoint red edges in $W$ and so we'd have a rainbow matching: $e$, the red edge in $W$ which is vertex independent from $e$, and the third color edge in the complex.

So $w_5$ has 2 non-red edges going to the complex. Clearly at most one of these goes to v; w.l.o.g. say one which does not go to $v$ is blue --- then it must be $(w_5, u_2)$ to avoid a rainbow matching. We now identified two vertex disjoint blue edges, $(v,u_1)$ and $(w_5,u_2)$, two disjoint red edges, $(w_1, w_2)$ and $(w_3, w_4)$, and one more red edge $(v,u_3)$. We claim that any green edge not blocking both blue edges can be extended to a rainbow matching. Indeed, if it blocks both red edges in $W$, then that green edge, $(v,u_3)$ and $(w_5,u_2)$ is  rainbow matching. If it does not block both red edges in $W$ and does not block the two blue edges (as we assume), then we can find a rainbow matching.

However, only $3$ green edges can block both blue edges, otherwise there was a green cycle. If there are no more green edges in $G \setminus \{v_j\}$ than these edges, then the the non-green degree of $v_j$ was at most $3$, contradicting that $v_j$ is a leaf in at most one of the trees (in $H$!). Therefore there is at least one green edge not blocking both blue edges, and thus, there is a rainbow matching.

\paragraph{Case 2: Exactly one disjoint edge within $W$.} There is at least one edge within W, but no pair of disjoint edges. Then these edges either form a star or a triangle. First suppose we have a star; w.l.o.g., say $w_1$ is the center of the star. We claim that among $w_2, w_3, w_4, w_5$ we can find $w_i, w_j$ so that $(w_1, w_i)$ is an edge of color red (say) and $w_j$ sends two non-red edges to the complex. Indeed, either the star with center $w_1$ contains $4$ leaves or there exists $w_j, j \in \{2, 3, 4, 5\}$ such that $(w_1, w_j)$ is not an edge. 

First, suppose each of $w_2, w_3, w_4, w_5$ have an edge going to $w_1$. Then by the Pigeonhole Principle, we can find $w_i, w_j$ among them such that $(w_1, w_i)$ and $(w_1, w_j)$ are the same color (say red). Notice then that $w_j$ has at least two non-red edges by the no-common-leaves condition, and both of these must go to the complex since $w_1$ is the center of our star.

Suppose on the other hand that $\exists w_j, j \in \{2, 3, 4, 5\}$ such that $(w_1, w_j)$ is not an edge. Then fix this $w_j$ and choose $w_i$ to be such that $(w_1, w_i)$ is an edge (which is possible since there's at least one edge within W and $w_1$ is the center of our star). Let the color of $(w_1, w_i)$ be red w.l.o.g. and notice that $w_j$ must send at least two non-red edges to the complex. Thus the claim is true.

We can say w.l.o.g. that $i=2$, $j=3$. That is, $(w_1, w_2)$ is a red edge and $w_3$ sends two non-red edges to the complex. In particular $w_3$ sends a non-red edge to the complex that does not go to $v$; w.l.o.g. say it is blue. Then it must be $(w_3, u_2)$ to prevent a rainbow matching. Now consider $w_4$ and $w_5$: notice that if either one has a green edge, we are done. Therefore we may assume they both have green-degree 0. 

Fix the green edge $(v, u_2)$, meaning we will build a rainbow matching containing this edge. There are no edges between $w_4$ and $w_5$ since $w_1$ is the center of our star so then $w_4$ and $w_5$ have at least 4 distinct red and 4 distinct blue edges between them. By no double edges, at most 4 of these 8 edges -- at most 2 from each of $w_4$ and $w_5$ -- cover the green edge $(v, u_2)$, and since we have no monochromatic cycles, at most 3 of either color do. Therefore we can, w.l.o.g., choose a red edge from $w_4$ and a blue edge from $w_5$ which are both disjoint from the green $(v, u_2)$. Then they must have the same endpoint, or else we have a rainbow matching. Moreover, this endpoint must be $w_1$ because otherwise we could choose the red $(w_1, w_2)$ along with the blue edge from $w_5$ which is disjoint from the green $(v, u_2)$ and we would have a rainbow matching. Thus we can assume we have red $(w_4, w_1)$ and blue $(w_5, w_1)$. 

Now $w_4$ and $w_5$ each have at least three more edges. These edges must avoid vertices in $W$ because the edges within $W$ were assumed to be a star (centred at $w_1$). In particular, they must each have one edge which avoids both $v$ and $u_2$. If either of these is blue, we are done along with the red $(w_1, w_2)$; on the other hand, if the edge from say $w_4$ is red, then we take it along with the blue $(w_5, w_1)$. So in any case, we can find a rainbow matching, and the star case is complete.

Now suppose we have a triangle within W: we can say w.l.o.g. that it is between the vertices $w_1, w_2, w_3$. Then all edges from $w_4$ and $w_5$ go to the complex. Clearly the edges in the triangle cannot all be the same color because that would make a cycle. If we have one edge of each color, then we claim we are done easily: w.l.o.g., say we have red $(w_1, w_2)$, blue $(w_1, w_3)$, and green $(w_2, w_3)$. Then choose any edge from $w_4$ which does not go to $v$. We can say w.l.o.g. that it is green. It blocks at most one of the red $(v, u_3)$ and the blue $(v, u_1)$, so we can pick one of these along with our green edge, and then complete our rainbow matching with an edge from the triangle.

So we may now assume that we have two edges in the triangle of one color, and the third edge is a different color. We can say w.l.o.g. that $(w_1, w_2)$ and $(w_2, w_3)$ are red and $(w_1, w_3)$ is blue. Now notice that $w_4$ and $w_5$ each send at least 2 non-red edges to the complex. If all four of these are blue, then at least one is disjoint from the green $(v, u_2)$ (otherwise we'd have a cycle), and we're done. So then we may assume at least one of these edges is green. If this green edge does not go to $v$, it is disjoint from either the red $(v, u_3)$ or the blue $(v, u_1)$, and we finish the rainbow matching with an appropriate edge from the triangle. So we may assume the green edge goes to $v$: w.l.o.g., say it is $(w_4, v)$. Now look at any non-red edge from $w_5$ which doesn't go to $v$. If it's green, we're done, as argued above. And if it's blue, then we take it along with the green $(w_4, v)$ and a red edge from the triangle. So in either case we have a rainbow matching, and the triangle case is complete.

\paragraph{Case 3: No edges within $W$.}
Observe that if any edge from a vertex in $W$, say $w_1$ were to go to some entirely new vertex $w$, then we could replace say $w_2$ with $w$ to get back to Case 2 (or possibly Case 1). So we may assume that all edges from $W$ go to the complex. That is, ignoring the edges within the complex, we have a bipartite graph with $W$ on one side and the complex on the other. We claim that this is in fact a complete bipartite graph. Indeed, observe that each vertex in $W$ sends four vertices to the complex. Since it cannot send two edges to the same vertex, it must send one edge each to each vertex in the complex. This describes a complete bipartite graph.

Together with the edges in the complex, there are $23$ edges. By the Pigeonhole Principle, there is a color, w.l.o.g. say red, with at least $8$ edges on these $9$ vertices. Notice, that it cannot be more, because then there would be a cycle. Thus, there are exaclty $8$ red edges, including the red one in the complex. For similar reasons, the remaning $15$ edges is split $8$ and $7$ between the remaining two colors. Therefore, in the bipartite complete graph, we can say w.l.o.g. that there are $7$ red and $7$ blue edges and there are $6$ green edges.

Take any of the green edges. The remaining $K_{4,3}$ complete graph contains 
$12$ edges, at most $5$ of them are green. So there are at least $7$ non-green 
edges. Not all can be red, otherwise there would be cycles. We can assume w.l.o.g.
that there are more red edges than blue ones, so there are at least $4$
red edges. There are the following cases
\begin{enumerate}
\item There is only one blue edge, but then there are at least $6$ red
edges. Only $2$ red edges can block one of its vertices, and only $3$ the
other one, so there should be vertex independent red edge, we are
ready.

\item There are two or three blue edges, they share a common vertex. The
shared vertex can be blocked by at most two red edges. There are more
red edges, which block at most one of the blue edges. We are ready.

\item There are a pair of blue edges, vertex disjoint. Only 2 red edges
can  block both of them, however, there are at least 4 red edges, so we are ready.
\end{enumerate}\qed
\end{proof}

The following two lemmas establish the base cases of the induction. The first lemma is stated and proved for an arbitrary number of path degree sequences, later we use that version in a proof.

\begin{lemma}\label{lem:hamilton}
Let $D_1, D_2,\ldots D_k$ be path degree sequences without common leaves. They have edge disjoint realizations.
\end{lemma}
\begin{proof}
The proof is by construction. It should be clear that $n \ge 2k$, since any tree contains at least two leaves.
We can say, w.l.o.g. that the leaves in the $i$ path have indexes $i$ and $\left\lceil\frac{n}{2}\right\rceil+i$. Then the  $i\th$ path contains the edges
$(i,n-1+i)$, $(n-1+i,1+i)$, $(1+i,n-2+i)$, $(n-2+i,2+i)$, $\ldots$, where the indexes are modulo $n$ shifted by $1$, that is, between $1$ and $n$.
The last edge is
 $\left(\left\lceil\frac{n}{2}\right\rceil+i-1,\left\lceil\frac{n}{2}\right\rceil+i\right)$ 
if $n$ is even, and  $\left(\left\lceil\frac{n}{2}\right\rceil+i+1,\left\lceil\frac{n}{2}\right\rceil+i\right)$ if $n$ is odd. Figure~\ref{fig:octagon} shows an example for $8$ vertices. It is easy to see that there are no parallel edges if such a path is rotated with at most $\left\lfloor\frac{n}{2}\right\rfloor$ vertices.
\qed
\end{proof}

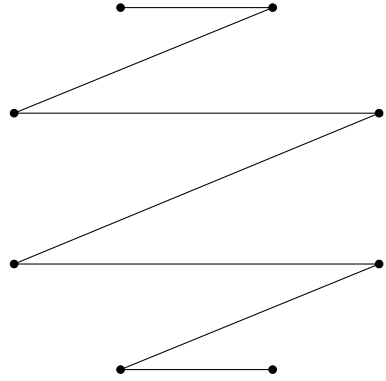
\begin{figure}
\setlength{\unitlength}{1cm}
\begin{center}
	\begin{tikzpicture}
\draw (4,0) -- (2,0) -- (5.4,1.4) -- (0.6,1.4)  -- (5.4,3.4) -- (0.6,3.4)  -- (4,4.8) -- (2,4.8) ;
\draw[black, fill= black] (2,0) circle(0.05);
\draw[black, fill= black] (4,0) circle(0.05);
\draw[black, fill= black] (0.6,1.4) circle(0.05);
\draw[black, fill= black] (5.4,1.4) circle(0.05);
\draw[black, fill= black] (0.6,3.4) circle(0.05);
\draw[black, fill= black] (5.4,3.4) circle(0.05);
\draw[black, fill= black] (2,4.8) circle(0.05);
\draw[black, fill= black] (4,4.8) circle(0.05);
\end{tikzpicture}
\end{center}
\caption{An example Hamiltonian path on 8 vertices. See the text for details.}\label{fig:octagon}
\end{figure}

\begin{lemma}\label{lem:smallcases}
Let $D_1, D_2, D_3, D_4$ be tree degree sequences on at most $10$ vertices, without common leaves. They have edge disjoint realizations.
\end{lemma}
\begin{proof}
Up to isomorphism, there are only $14$ possible such degree sequence quartets. The appendix contains a realization for each of them.
\qed
\end{proof}

Now we are ready to prove the main theorem.

\begin{theorem}
Let $D_1, D_2, D_3, D_4$ be tree degree sequences without common leaves. They have edge disjoint tree realizations.
\end{theorem}
\begin{proof}
The proof is by induction, the base cases are the degree sequences on at most $10$ vertices and the path degree sequence quartets. They all have edge disjoint realizations, based on Lemmas~\ref{lem:hamilton}~and~\ref{lem:smallcases}.

So assume that $D_1, D_2, D_3, D_4$ are tree degree sequences on more than $10$ vertices and at least one of them is not a path degree sequence. Then there exist vertices $v$ and $w$ and an index $i$ such that $d_v^{(i)} =1$ for all $j \ne i$, $d_v^{j} =2$ and $d_w^{(i)} > 2$, according to Lemma~\ref{lem:reduction}. Consider the degree sequences $D_1', D_2', D_3', D_4'$ wich is obtained by deleting vertex $v$ and subtracting $1$ from $d_w^{(i)}$. These are tree degree sequences without common leaves, and based on the inductive assumption, they have edge disjoint realizations. Let $H$ be the colored graph representing these edge disjoint realizations and permute the degree sequences (and the colors accordingly) that $D_i$ is moved to the fourth position. Let $G$ be the subgraph of $H$ containing the first $3$ colors after the beforementioned permutation. Since $G$ contains at least $10$ vertices, $G \setminus \{w\}$ contains a rainbow matching according to Lemma~\ref{lem:rainbow}. Let $(v_1,v_2)$, $(v_3,v_4)$ and $(v_5,v_6)$ denote the edges in the rainbow matching. The realization of $D_1, D_2, D_3$ and $D_4$ is obtained by the following way. Take the realization representd by $H$. Add vertex $v$. Connect $v$ with $w$ in the first tree, delete the edges of the rainbow matching, and connect $v$ to all the vertices incident to the edges of the rainbow matching, $2$ edges for each tree, according to the color of the deleted edge.
\qed
\end{proof}

\section{Some results in the general case}

We now present some results in the general case, i.e.\ for an arbitrary $k$ number of tree degree sequences. First we show that $n \geq 4k - 2$ suffices to guarantee a rainbow matching. However for our original purpose of finding edge-disjoint realizations via the inductive proof, this is not sufficient to show that the induction step goes through every time, since our base case is $n = 2k$. We need something else to bridge the gap between $n = 2k$ and $n = 4k - 2$. This is accomplished by our second characterization, which adds an extra condition and says that if we have at least $2k-4$ vertices that are not leaves in any tree, then we are indeed guaranteed edge-disjoint realizations.

\subsection{Rainbow matchings from matchings: $n = O(k)$ guarantees a rainbow matching}

We now show that $n = O(k)$ suffices to guarantee a rainbow matching. The broad line of attack will be to stitch a rainbow matching together from regular (singly-colored, and large but not necessarily perfect) matchings. A crucial ingredient in guaranteeing large matchings will be the fact that a tree with $m$ non-leaves must contain a matching roughly of size $m/2$. The idea will be that when $n$ is large enough, the no-common-leaves condition guarantees a large number of non-leaves in each color, which then guarantees large matchings in each color, which can then be stitched together into a rainbow matching. We formalize the main ingredients as the following lemmas.

\begin{lemma} \label{lem:greedy-rainbow}
Let $G$ be an edge colored graph such that for each color $c_i$, $i = 1, 2, \ldots k$
there is a matching of size $2i$ in the subgraph of color $c_i$. Let $v$ be an arbitrary vertex. Then $G \setminus \{v\}$ has a rainbow matching of size $k$.
%
\end{lemma}
\begin{proof}
The proof is by induction using the Pigeonhole Principle. Since there are $2$ disjoint edges of the first color in $G$, at least one of them is not incident to $v$. Take that edge, which will be in the rainbow matching.

Assume that we already found a rainbow matching of size $i$. There is a matching of size $2i+2$ in the subgraph of color $c_{i+1}$. At most $2i$ of them are blocked by the rainbow matching of size $i$, and at most one of them is incident to $v$. Thus, there is an edge of color $i+1$ which is disjoint from the rainbow matching of size $i$ and not incident to $v$. Extend the rainbow matching with this edge. 
\qed
\end{proof}

\begin{lemma} \label{lem:matching-internal}
A tree with at least one edge and $m$ internal nodes contains a matching of size at least $\left\lceil\frac{m+1}{2}\right\rceil$.
\end{lemma}
\begin{proof}
The proof is by induction. The base cases are the trees with $2$ and $3$ vertices. They have $0$ and $1$ internal nodes (i.e.\ non-leaves) respectively, and they each have an edge, which is a matching of size $1$.

Now assume that the number of vertices in tree $T$ is more than $3$, and the number of internal nodes in it is $m$. Take any leaf and its incident edge $e$. There are two cases.
\begin{enumerate}
\item The non-leaf vertex of $e$ has degree more than $2$. Then $T' = T \setminus \{e\}$ has the same number of internal nodes as $T$. By the inductive hypothesis,  $T'$ has a matching of size $\left\lceil\frac{m+1}{2}\right\rceil$, so $T$ does also.
\item The non-leaf vertex of $e$ has degree $2$. Let its other edge be denoted by $f$. Then the internal nodes in $T' = T \setminus \{e,f\}$ is the internal nodes in $T$ minus at most $2$. Thus $T'$ has a matching $M$ of size $\left\lceil\frac{m-1}{2}\right\rceil$. $M \cup \{e\}$ is a matching in $T$ with size $\left\lceil\frac{m+1}{2}\right\rceil$. 
\end{enumerate}
\qed
\end{proof}

We now show that $n \geq 4k - 2$ suffices to guarantee a rainbow matching.

\begin{theorem}\label{theo:lower-bound-vertices}
Let $k$ trees be given on $n$ vertices, $k \ge 5$, having no common leaves. Let $w$ be an arbitrary vertex. Then if the number of vertices are greater or equal than $4k - 2$, we can find a rainbow matching in the first $k-1$ trees avoiding $w$.
\end{theorem}
\begin{proof}
Arrange our $k-1$ trees in increasing order of number of internal nodes. We would 
like to show that the $i\th$ tree has a matching of size $2i$. This is sufficient to find a rainbow matching, according to Lemma~\ref{lem:greedy-rainbow}.

Since internal nodes are exactly the vertices of a tree which are not leaves, we have also arranged the trees in decreasing order of number of leaves. Each tree has at least $2$ leaves, therefore in the $k-1-i$ trees above the $i\th$ tree and in the $k\th$ tree there are altogether at least $2(k-i)$ leaves. Since no vertex is a leaf in more than one tree, there remain only at most $n - 2(k-i)$ vertices that might still be leaves in the the $i\th$ tree and the $i-1$ trees below. And since the number of leaves in the trees below is no less than in the $i\th$ tree, the $i\th$ tree contains at most
$$
\left\lfloor\frac{n - 2(k-i)}{i}\right\rfloor 
$$
leaves, and thus at least
$$
n-\left\lfloor\frac{n - 2(k-i)}{i}\right\rfloor = \left\lceil\frac{(i-1)n + 2(k-i)}{i}\right\rceil
$$
internal nodes. If $n \ge 4k - 2$, this means at least
$$
\left\lceil\frac{(i-1)(4k -2) + 2(k-i)}{i}\right\rceil = \left\lceil\frac{4ki - 2k - 4i +2}{i}\right\rceil = 4k - 4 - \left\lfloor\frac{2k-2}{i}\right\rfloor
$$
internal nodes. According to Lemma~\ref{lem:matching-internal}, there is a matching of a given lowerly bounded size that must exist in the $i\th$ tree, and we are going to show that
\begin{equation}
\left\lceil\frac{4k - 4 - \left\lfloor\frac{2k-2}{i}\right\rfloor+1}{2}\right\rceil \ge 2i.\label{eq:inequality}
\end{equation}
When $i = k-1$, the left hand side is
$$
\left\lceil\frac{4k-4-2+1}{2}\right\rceil =2(k-1) =2i.
$$
For $i<k-1$, it is sufficient to show that
$$
\frac{4k - 3 - \frac{2k-2}{i}}{2} \ge 2i.
$$
After rearranging, we get that
$$
0 \ge 4i^2 - (4k-3)i + 2k -2 
$$
Solving the second order equation, we get that
$$
\frac{4k-3-\sqrt{(4k-7)^2-8}}{8} \le i \le \frac{4k-3+\sqrt{(4k-7)^2-8}}{8}.
$$
Rounding the discriminant knowing that $k \ge 5$, we get that
$$
\frac{4k-3-(4k-8)}{8} \le i \le \frac{4k-3+4k-8}{8}.
$$
namely,
$$
\frac{5}{8} \le i \le k - \frac{11}{8}
$$
which holds since $1\le i \le k-2$. Therfore, in the $i\th$ tree there is a matching of size at least $2i$, which is sufficient to have the prescribed rainbow matching.
\qed
\end{proof}

\subsection{Edge-disjoint realizations under a condition on the degree distribution}

Theorem~\ref{theo:lower-bound-vertices} is not strong enough to prove the full theorem of edge-disjoint realizations, since in our inductive proof we need to find rainbow matchings at each inductive step, starting from $n = 2k$. But by adding an extra condition to the degree distribution, and showing that this condition is maintained throughout the induction process, we are successfully able to guarantee edge-disjoint realizations.

Define a \emph{never-leaf} to be a vertex that is not a leaf in any tree.

\begin{theorem}
$k$ tree degree sequences without common leaves and with at least $2k-4$ never-leaves always have edge-disjoint realizations.
\end{theorem}
\begin{proof}
We will use the same inductive proof as presented originally. The crucial observation about that proof is that nowhere during the inductive step do we create any new leaves in any tree. This means the number of never-leaves does not change during the inductive step, and so at each step we have at least $2k-4$ never-leaves.

It only remains to be shown, then, that whenever we have $2k-4$ never-leaves we can find a rainbow matching. We claim that in each tree there are at least $4k-6$ internal nodes. Indeed, the $2k-4$ never-leaves are certainly internal nodes in this tree. And in each of the other $k-1$ trees there are at least two leaves, and these leaves are internal nodes in all other trees because no common leaves, giving an additional $2k-2$ internal nodes, altogether $4k-6$ internal nodes. By Lemma~\ref{lem:matching-internal} this means we have matchings of size at least 
$$
\left\lceil\frac{4k-5}{2}\right\rceil = 2k-2
$$
in each tree, and by Lemma~\ref{lem:greedy-rainbow} these guarantee a rainbow matching, and we are done.
\qed
\end{proof}

\subsection{A conditional theorem and the $k = 5$ case}

The consequence of Theorem~\ref{theo:lower-bound-vertices} is the following.
\begin{theorem}\label{theo:conditional}
Fix a $k$. If all tree degree sequence $k$-tuples without common leaves on at most $4k-2$ vertices have edge disjoint tree realizations, then any tree degree sequence  $k$-tuples without common leaves have edge disjoint tree realizations.
\end{theorem}
\begin{proof}
The proof is by induction. The base cases are the path degree sequences, which have edge disjoint realizations, according to Lemma~\ref{lem:hamilton}, and the degree sequences on at most $4k-2$, which have edge disjoint realizations by the condition of the theorem.

Let $D_1, D_2, \ldots, D_k$ be tree degree sequences without common leaves on more than $4k-2$ vertices. By Lemma~\ref{lem:reduction}, there are vertices $v$, $w$ and index $i$, such that $d_v^{(i)} = 1$, for all $j \ne i$, $d_v^{(j)} =2$ and $d_w^{(i)} >2$. Construct the degree sequences $D_1', D_2', \ldots, D_k'$ by removing $v$ and subtracting $1$ from $d_w^{(i)}$. These are tree degree sequences on at least $4k-2$ vertices, and they have edge disjoint realizations $T_1', T_2',\ldots, T_k'$ by the inductive hypothesis. Furthermore, there is a rainbow matching on all the trees except the $i\th$ avoiding vertex $w$, according to Theorem~\ref{theo:lower-bound-vertices}. Construct a realization of $D_1, D_2, \ldots, D_k$ in the following way. Start with $T_1', T_2', \ldots T_k'$. Add vertex $v$, connect it to $w$ in $T_i'$. Delete the edges in the rainbow matching, and connect $v$ to their $2k-2$ vertices, two edges in each tree, according the color of the deleted edge.
\qed
\end{proof}

When $k = 5$, Theorem~\ref{theo:conditional} says the following: if all tree degree sequence quintets without common leaves and on at most $18$ vertices have edge disjoint tree realizations, then all tree degree sequence quintets have edge disjoint tree realizations. A computer-aided search showed that up to permutation of sequences and vertices, there are at most $592000$ tree degree quintets without common leaves and on at most 18 vertices, and they all have edge disjoint tree realizations.

\section*{Appendix}

Up to permutations of degree sequences and vertices, there are $14$ tree degree sequence quartets on at most  $10$ vertices without common leaves. This appendix gives an example realization for all of them.

If the number of vertices is $8$, there is only one possible degree sequence quartet, each degree sequence is a path degree sequence (case 1).

If the number of vertices is $9$, there are $2$ possible cases: either all degree squences are path degree sequences (case 2) or there is a degree $3$ (case 3).

If the number of vertices is $10$, there are $11$ possible cases. All degree sequences are path degree sequences (case 4), there is a degree 3 which might be on a vertex with a leaf (case 5) or without a leaf (case 6), there is a degree $4$ (case 7) or there are $2$ degree  $3$s in the degree sequences (cases 8-14).

The two $3$s might be in the same degree sequence, and the leaves on these two vertices might be in the same degree sequence (case 8) or in different degree sequences (case 9).

If the two degree $3$s are in different degree sequences, they might be on the same vertex (case 10) or on different vertices.

If the two degree $3$s are in different sequences, $D_i$ and $D_j$, and on different vertices $u$ and $v$, consider the degrees of $u$ and $v$ in $D_i$ and $D_j$ which are not $3$. They might be both $1$ (case 11), or else maybe one of them is $1$ and the other is $2$ (case 12), or else both of them are $2$. In this latter case, the degree $1$s on $u$ and $v$ might be in the same degree sequence (case 13) or in different degree sequences (case 14).

The realizations are represented with an adjacency matrix, in which $0$ denotes the absence of edges, and for each degree sequence $D_i$, $i$ denotes the edges in the realization of $D_i$.

\begin{enumerate}
\item  
\begin{eqnarray}
D_1& = &1, 2, 2, 2, 1, 2, 2, 2 \nonumber\\
D_2& = &2, 1, 2, 2, 2, 1, 2, 2 \nonumber\\
D_3& = &2, 2, 1, 2, 2, 2, 1, 2 \nonumber\\
D_4& = &2, 2, 2, 1, 2, 2, 2, 1 \nonumber
\end{eqnarray}
$$
\left(
\begin{array}{cccccccc}
0& 1& 2& 2& 3& 3& 4& 4 \\
1& 0& 2& 3& 3& 4& 4& 1 \\
2& 2& 0& 3& 4& 4& 1& 1 \\
2& 3& 3& 0& 4& 1& 1& 2 \\
3& 3& 4& 4& 0& 1& 2& 2 \\
3& 4& 4& 1& 1& 0& 2& 3 \\
4& 4& 1& 1& 2& 2& 0& 3 \\
4& 1& 1& 2& 2& 3& 3& 0 
\end{array}
\right)
$$

\item 
\begin{eqnarray}
D_1& = &1, 2, 2, 2, 2, 1, 2, 2, 2 \nonumber\\
D_2& = &2, 1, 2, 2, 2, 2, 1, 2, 2 \nonumber\\
D_3& = &2, 2, 1, 2, 2, 2, 2, 1, 2 \nonumber\\
D_4& = &2, 2, 2, 1, 2, 2, 2, 2, 1 \nonumber
\end{eqnarray}
$$
\left(
\begin{array}{ccccccccc}
0& 1& 2& 2& 3& 3& 4& 4& 0 \\
1& 0& 2& 3& 3& 4& 4& 0& 1 \\
2& 2& 0& 3& 4& 4& 0& 1& 1 \\
2& 3& 3& 0& 4& 0& 1& 1& 2 \\
3& 3& 4& 4& 0& 1& 1& 2& 2 \\
3& 4& 4& 0& 1& 0& 2& 2& 3 \\
4& 4& 0& 1& 1& 2& 0& 3& 3 \\
4& 0& 1& 1& 2& 2& 3& 0& 4 \\
0& 1& 1& 2& 2& 3& 3& 4& 0 
\end{array}
\right)
$$

\item 
\begin{eqnarray}
D_1& = &1, 3, 2, 2, 1, 2, 2, 2, 1 \nonumber\\
D_2& = &2, 1, 2, 2, 2, 1, 2, 2, 2 \nonumber\\
D_3& = &2, 2, 1, 2, 2, 2, 1, 2, 2 \nonumber\\
D_4& = &2, 2, 2, 1, 2, 2, 2, 1, 2 \nonumber
\end{eqnarray}
$$
\left(
\begin{array}{ccccccccc}
0& 1& 0& 2& 3& 3& 4& 4& 2 \\
1& 0& 2& 3& 3& 4& 4& 1& 1 \\
0& 2& 0& 3& 4& 4& 1& 1& 2 \\
2& 3& 3& 0& 0& 1& 1& 2& 4 \\
3& 3& 4& 0& 0& 1& 2& 2& 4 \\
3& 4& 4& 1& 1& 0& 2& 0& 3 \\
4& 4& 1& 1& 2& 2& 0& 3& 0 \\
4& 1& 1& 2& 2& 0& 3& 0& 3 \\
2& 1& 2& 4& 4& 3& 0& 3& 0 
\end{array}
\right)
$$

\item 
\begin{eqnarray}
D_1& = &1, 2, 2, 2, 2, 1, 2, 2, 2, 2 \nonumber\\
D_2& = &2, 1, 2, 2, 2, 2, 1, 2, 2, 2 \nonumber\\
D_3& = &2, 2, 1, 2, 2, 2, 2, 1, 2, 2 \nonumber\\
D_4& = &2, 2, 2, 1, 2, 2, 2, 2, 1, 2 \nonumber
\end{eqnarray}
$$
\left(
\begin{array}{cccccccccc}
0& 1& 2& 2& 3& 3& 4& 4& 0& 0 \\
1& 0& 2& 3& 3& 4& 4& 0& 0& 1 \\
2& 2& 0& 3& 4& 4& 0& 0& 1& 1 \\
2& 3& 3& 0& 4& 0& 0& 1& 1& 2 \\
3& 3& 4& 4& 0& 0& 1& 1& 2& 2 \\
3& 4& 4& 0& 0& 0& 1& 2& 2& 3 \\
4& 4& 0& 0& 1& 1& 0& 2& 3& 3 \\
4& 0& 0& 1& 1& 2& 2& 0& 3& 4 \\
0& 0& 1& 1& 2& 2& 3& 3& 0& 4 \\
0& 1& 1& 2& 2& 3& 3& 4& 4& 0 
\end{array}
\right)
$$

\item 
\begin{eqnarray}
D_1& = &1, 3, 2, 2, 2, 1, 2, 2, 2, 1 \nonumber\\
D_2& = &2, 1, 2, 2, 2, 2, 1, 2, 2, 2 \nonumber\\
D_3& = &2, 2, 1, 2, 2, 2, 2, 1, 2, 2 \nonumber\\
D_4& = &2, 2, 2, 1, 2, 2, 2, 2, 1, 2 \nonumber
\end{eqnarray}
$$
\left(
\begin{array}{cccccccccc}
0& 1& 0& 2& 3& 3& 4& 4& 0& 2 \\
1& 0& 2& 3& 3& 4& 4& 0& 1& 1 \\
0& 2& 0& 3& 4& 4& 0& 1& 1& 2 \\
2& 3& 3& 0& 0& 0& 1& 1& 2& 4 \\
3& 3& 4& 0& 0& 1& 1& 2& 2& 4 \\
3& 4& 4& 0& 1& 0& 2& 2& 3& 0 \\
4& 4& 0& 1& 1& 2& 0& 0& 3& 3 \\
4& 0& 1& 1& 2& 2& 0& 0& 4& 3 \\
0& 1& 1& 2& 2& 3& 3& 4& 0& 0 \\
2& 1& 2& 4& 4& 0& 3& 3& 0& 0 
\end{array}
\right)
$$

\item 
\begin{eqnarray}
D_1& = &1, 2, 2, 2, 3, 1, 2, 2, 2, 1 \nonumber\\
D_2& = &2, 1, 2, 2, 2, 2, 1, 2, 2, 2 \nonumber\\
D_3& = &2, 2, 1, 2, 2, 2, 2, 1, 2, 2 \nonumber\\
D_4& = &2, 2, 2, 1, 2, 2, 2, 2, 1, 2 \nonumber
\end{eqnarray}
$$
\left(
\begin{array}{cccccccccc}
0& 1& 0& 2& 3& 3& 4& 4& 0& 2 \\
1& 0& 2& 0& 3& 4& 4& 0& 1& 3 \\
0& 2& 0& 3& 4& 4& 0& 1& 1& 2 \\
2& 0& 3& 0& 4& 0& 1& 1& 2& 3 \\
3& 3& 4& 4& 0& 1& 1& 2& 2& 1 \\
3& 4& 4& 0& 1& 0& 2& 2& 3& 0 \\
4& 4& 0& 1& 1& 2& 0& 3& 3& 0 \\
4& 0& 1& 1& 2& 2& 3& 0& 0& 4 \\
0& 1& 1& 2& 2& 3& 3& 0& 0& 4 \\
2& 3& 2& 3& 1& 0& 0& 4& 4& 0 
\end{array}
\right)
$$

\item 
\begin{eqnarray}
D_1& = &1, 4, 2, 2, 1, 2, 2, 2, 1, 1 \nonumber\\
D_2& = &2, 1, 2, 2, 2, 1, 2, 2, 2, 2 \nonumber\\
D_3& = &2, 2, 1, 2, 2, 2, 1, 2, 2, 2 \nonumber\\
D_4& = &2, 2, 2, 1, 2, 2, 2, 1, 2, 2 \nonumber
\end{eqnarray}
$$
\left(
\begin{array}{cccccccccc}
0& 1& 0& 0& 3& 3& 4& 4& 2& 2 \\
1& 0& 2& 3& 3& 4& 4& 1& 1& 1 \\
0& 2& 0& 3& 0& 4& 1& 1& 2& 4 \\
0& 3& 3& 0& 0& 1& 1& 2& 4& 2 \\
3& 3& 0& 0& 0& 1& 2& 2& 4& 4 \\
3& 4& 4& 1& 1& 0& 2& 0& 3& 0 \\
4& 4& 1& 1& 2& 2& 0& 0& 0& 3 \\
4& 1& 1& 2& 2& 0& 0& 0& 3& 3 \\
2& 1& 2& 4& 4& 3& 0& 3& 0& 0 \\
2& 1& 4& 2& 4& 0& 3& 3& 0& 0 
\end{array}
\right)
$$

\item 
\begin{eqnarray}
D_1& = &1, 3, 2, 2, 1, 3, 2, 2, 1, 1 \nonumber\\
D_2& = &2, 1, 2, 2, 2, 1, 2, 2, 2, 2 \nonumber\\
D_3& = &2, 2, 1, 2, 2, 2, 1, 2, 2, 2 \nonumber\\
D_4& = &2, 2, 2, 1, 2, 2, 2, 1, 2, 2 \nonumber
\end{eqnarray}
$$
\left(
\begin{array}{cccccccccc}
0& 1& 0& 2& 0& 3& 4& 4& 2& 3 \\
1& 0& 0& 3& 3& 4& 4& 1& 1& 2 \\
0& 0& 0& 3& 4& 4& 1& 1& 2& 2 \\
2& 3& 3& 0& 0& 1& 1& 2& 0& 4 \\
0& 3& 4& 0& 0& 1& 2& 2& 4& 3 \\
3& 4& 4& 1& 1& 0& 2& 0& 3& 1 \\
4& 4& 1& 1& 2& 2& 0& 3& 0& 0 \\
4& 1& 1& 2& 2& 0& 3& 0& 3& 0 \\
2& 1& 2& 0& 4& 3& 0& 3& 0& 4 \\
3& 2& 2& 4& 3& 1& 0& 0& 4& 0 
\end{array}
\right)
$$

\item 
\begin{eqnarray}
D_1& = &1, 3, 3, 2, 1, 2, 2, 2, 1, 1 \nonumber\\
D_2& = &2, 1, 2, 2, 2, 1, 2, 2, 2, 2 \nonumber\\
D_3& = &2, 2, 1, 2, 2, 2, 1, 2, 2, 2 \nonumber\\
D_4& = &2, 2, 2, 1, 2, 2, 2, 1, 2, 2 \nonumber
\end{eqnarray}
$$
\left(
\begin{array}{cccccccccc}
0& 1& 0& 0& 3& 3& 4& 4& 2& 2 \\
1& 0& 2& 3& 3& 0& 4& 1& 1& 4 \\
0& 2& 0& 3& 4& 4& 1& 1& 2& 1 \\
0& 3& 3& 0& 0& 1& 1& 2& 4& 2 \\
3& 3& 4& 0& 0& 1& 2& 2& 4& 0 \\
3& 0& 4& 1& 1& 0& 2& 0& 3& 4 \\
4& 4& 1& 1& 2& 2& 0& 0& 0& 3 \\
4& 1& 1& 2& 2& 0& 0& 0& 3& 3 \\
2& 1& 2& 4& 4& 3& 0& 3& 0& 0 \\
2& 4& 1& 2& 0& 4& 3& 3& 0& 0 
\end{array}
\right)
$$

\item 
\begin{eqnarray}
D_1& = &1, 3, 2, 2, 1, 2, 2, 2, 1, 2 \nonumber\\
D_2& = &2, 1, 2, 2, 2, 1, 2, 2, 2, 2 \nonumber\\
D_3& = &2, 3, 1, 2, 2, 2, 1, 2, 2, 1 \nonumber\\
D_4& = &2, 2, 2, 1, 2, 2, 2, 1, 2, 2 \nonumber
\end{eqnarray}
$$
\left(
\begin{array}{cccccccccc}
0& 1& 0& 2& 3& 3& 4& 0& 2& 4 \\
1& 0& 2& 3& 3& 4& 4& 1& 1& 3 \\
0& 2& 0& 3& 4& 4& 1& 1& 2& 0 \\
2& 3& 3& 0& 0& 0& 1& 2& 4& 1 \\
3& 3& 4& 0& 0& 1& 0& 2& 4& 2 \\
3& 4& 4& 0& 1& 0& 2& 0& 3& 1 \\
4& 4& 1& 1& 0& 2& 0& 3& 0& 2 \\
0& 1& 1& 2& 2& 0& 3& 0& 3& 4 \\
2& 1& 2& 4& 4& 3& 0& 3& 0& 0 \\
4& 3& 0& 1& 2& 1& 2& 4& 0& 0 
\end{array}
\right)
$$

\item 
\begin{eqnarray}
D_1& = &1, 3, 2, 2, 1, 2, 2, 2, 1, 2 \nonumber\\
D_2& = &3, 1, 2, 2, 2, 1, 2, 2, 2, 1 \nonumber\\
D_3& = &2, 2, 1, 2, 2, 2, 1, 2, 2, 2 \nonumber\\
D_4& = &2, 2, 2, 1, 2, 2, 2, 1, 2, 2 \nonumber
\end{eqnarray}
$$
\left(
\begin{array}{cccccccccc}
0& 1& 0& 2& 3& 3& 4& 4& 2& 2 \\
1& 0& 2& 3& 3& 4& 4& 1& 1& 0 \\
0& 2& 0& 3& 0& 4& 1& 1& 2& 4 \\
2& 3& 3& 0& 0& 0& 1& 2& 4& 1 \\
3& 3& 0& 0& 0& 1& 2& 2& 4& 4 \\
3& 4& 4& 0& 1& 0& 2& 0& 3& 1 \\
4& 4& 1& 1& 2& 2& 0& 0& 0& 3 \\
4& 1& 1& 2& 2& 0& 0& 0& 3& 3 \\
2& 1& 2& 4& 4& 3& 0& 3& 0& 0 \\
2& 0& 4& 1& 4& 1& 3& 3& 0& 0 
\end{array}
\right)
$$

\item 
\begin{eqnarray}
D_1& = &1, 3, 2, 2, 1, 2, 2, 2, 1, 2 \nonumber\\
D_2& = &2, 1, 3, 2, 2, 1, 2, 2, 2, 1 \nonumber\\
D_3& = &2, 2, 1, 2, 2, 2, 1, 2, 2, 2 \nonumber\\
D_4& = &2, 2, 2, 1, 2, 2, 2, 1, 2, 2 \nonumber
\end{eqnarray}
$$
\left(
\begin{array}{cccccccccc}
0& 0& 0& 2& 3& 3& 4& 4& 2& 1 \\
0& 0& 2& 3& 3& 4& 4& 1& 1& 1 \\
0& 2& 0& 3& 4& 4& 1& 1& 2& 2 \\
2& 3& 3& 0& 0& 1& 1& 2& 0& 4 \\
3& 3& 4& 0& 0& 1& 2& 2& 4& 0 \\
3& 4& 4& 1& 1& 0& 2& 0& 3& 0 \\
4& 4& 1& 1& 2& 2& 0& 0& 0& 3 \\
4& 1& 1& 2& 2& 0& 0& 0& 3& 3 \\
2& 1& 2& 0& 4& 3& 0& 3& 0& 4 \\
1& 1& 2& 4& 0& 0& 3& 3& 4& 0 
\end{array}
\right)
$$

\item
\begin{eqnarray}
D_1& = &1, 3, 2, 2, 1, 2, 2, 2, 1, 2 \nonumber\\
D_2& = &2, 1, 2, 2, 2, 1, 2, 2, 2, 2 \nonumber\\
D_3& = &2, 2, 1, 2, 2, 3, 1, 2, 2, 1 \nonumber\\
D_4& = &2, 2, 2, 1, 2, 2, 2, 1, 2, 2 \nonumber
\end{eqnarray}
$$
\left(
\begin{array}{cccccccccc}
0& 0& 0& 2& 3& 3& 4& 4& 2& 1 \\
0& 0& 2& 3& 3& 4& 4& 1& 1& 1 \\
0& 2& 0& 3& 4& 4& 1& 1& 2& 0 \\
2& 3& 3& 0& 0& 1& 1& 2& 0& 4 \\
3& 3& 4& 0& 0& 1& 0& 2& 4& 2 \\
3& 4& 4& 1& 1& 0& 2& 0& 3& 3 \\
4& 4& 1& 1& 0& 2& 0& 3& 0& 2 \\
4& 1& 1& 2& 2& 0& 3& 0& 3& 0 \\
2& 1& 2& 0& 4& 3& 0& 3& 0& 4 \\
1& 1& 0& 4& 2& 3& 2& 0& 4& 0 
\end{array}
\right)
$$

\item 
\begin{eqnarray}
D_1& = &1, 3, 2, 2, 1, 2, 2, 2, 1, 2 \nonumber\\
D_2& = &2, 1, 2, 2, 2, 1, 2, 2, 2, 2 \nonumber\\
D_3& = &2, 2, 1, 3, 2, 2, 1, 2, 2, 1 \nonumber\\
D_4& = &2, 2, 2, 1, 2, 2, 2, 1, 2, 2 \nonumber
\end{eqnarray}
$$
\left(
\begin{array}{cccccccccc}
0& 0& 0& 2& 3& 3& 4& 4& 2& 1 \\
0& 0& 2& 3& 3& 4& 4& 1& 1& 1 \\
0& 2& 0& 3& 4& 0& 1& 1& 2& 4 \\
2& 3& 3& 0& 0& 1& 1& 2& 4& 3 \\
3& 3& 4& 0& 0& 1& 0& 2& 4& 2 \\
3& 4& 0& 1& 1& 0& 2& 0& 3& 4 \\
4& 4& 1& 1& 0& 2& 0& 3& 0& 2 \\
4& 1& 1& 2& 2& 0& 3& 0& 3& 0 \\
2& 1& 2& 4& 4& 3& 0& 3& 0& 0 \\
1& 1& 4& 3& 2& 4& 2& 0& 0& 0 
\end{array}
\right)
$$

\end{enumerate}

\begin{acknowledgements}
IM is supported by NKFIH Funds No. K116769 and No. SNN-117879.
\end{acknowledgements}



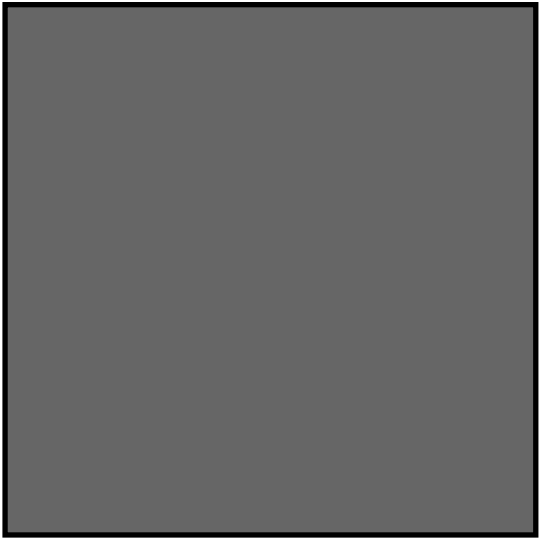
\end{document}